\documentclass[11pt,a4paper]{amsart}
\usepackage[utf8]{inputenc}
\usepackage{amsmath, amsfonts, amssymb, amsthm, amsopn}
\usepackage{stmaryrd}
\usepackage{url}
\usepackage[cmtip, all]{xy}

\usepackage{xcolor}

\usepackage{graphicx}
\usepackage{verbatim}

\usepackage{hyperref}

\newcommand{\eq}[2]{\begin{equation}\label{#1}#2 \end{equation}}
\newcommand{\ml}[2]{\begin{multline}\label{#1}#2 \end{multline}}
\newcommand{\ga}[2]{\begin{gather}\label{#1}#2 \end{gather}}
\newcommand{\et}{{\rm{\acute{e}t}}}
\newcommand{\ab}{\mathrm{ab}}

\theoremstyle{plain}
\newtheorem{theorem}{Theorem}[section]
\newtheorem{proposition}[theorem]{Proposition}
\newtheorem{corollary}[theorem]{Corollary}
\newtheorem{lemma}[theorem]{Lemma}

\theoremstyle{definition}
\newtheorem{definition}[theorem]{Definition}

\newtheorem{remark}[theorem]{Remark}
\newtheorem{question}[theorem]{Question}

\newtheorem{ass}[theorem]{Assumption}
\newtheorem{data}[theorem]{Data}

\newcommand{\CC}{\mathbb{C}}

\newcommand{\QQ}{\mathbb{Q}}
\newcommand{\RR}{\mathbb{R}}

\newcommand{\ZZ}{\mathbb{Z}}
\newcommand{\FF}{\mathbb{F}}

\newcommand{\sF}{\mathcal{F}}
\newcommand{\sV}{\mathcal{V}}

\newcommand{\tr}{\mathrm{tr}}

\newcommand{\sO}{\mathcal{O}}

\newcommand{\Mg}{{\rm{Char}}}
\newcommand{\cont}{{\rm cont}}
\newcommand{\QQp}{\overline{\QQ}_p}
\newcommand{\QQpp}{\overline{\QQ}_{p'}}
\newcommand{\ZZp}{\overline{\ZZ}_p}

\renewcommand{\setminus}{\smallsetminus}

\DeclareMathOperator{\End}{End}

\DeclareMathOperator{\Hom}{Hom{}}
\DeclareMathOperator{\Spec}{Spec}

\let\lim=\relax
\DeclareMathOperator*{\lim}{lim}
\let\dim=\relax
\DeclareMathOperator*{\im}{im}
\DeclareMathOperator*{\dim}{dim}

\DeclareMathOperator*{\colim}{colim}

\begin{document}

\title[Moduli spaces rank one]{Arithmetic  subspaces of moduli spaces of rank one local systems }
\author{H\'el\`ene Esnault \and Moritz Kerz}
\address{Freie Universit\"at Berlin, Arnimallee 3, 14195, Berlin,  Germany}
\email{esnault@math.fu-berlin.de}
\address{   Fakult\"at f\"ur Mathematik \\
Universit\"at Regensburg \\
93040 Regensburg, Germany}
\email{moritz.kerz@mathematik.uni-regensburg.de}
\thanks{The first  author is supported by  the Einstein program,  the second author by the SFB 1085 Higher Invariants}

\begin{abstract} We show that closed subsets of the character variety of a complex variety
  with negatively weighted  first homology, which are $p$-adically integral
  and  stabilized by an arithmetic Galois action, are motivic.
\end{abstract}

\maketitle

\section{Introduction}\label{sec:intro}

Let $F \subset \CC$ be a subfield of the field of complex numbers, $\overline F$ be its
algebraic closure in $\CC$, $X$ be a separated  geometrically connected scheme of finite
type over $F$. Let $\Pi$ be the abelianized fundamental group $\pi_1^{\ab}(X(\CC))=H_1(X(\CC)
,\ZZ )$. The $\QQ$-vector space 
$\Pi_\QQ=\Pi\otimes_\ZZ \QQ$ is endowed with a mixed Hodge structure with weights in
$\{-2,-1,0\}$.
The weights of $\Pi_\QQ$ are in $\{-2,-1\}$ if and only if for any open dense subscheme
$U\subset X$ the map $\pi_1^\ab(U(\CC))_\QQ\to \Pi_\QQ$ is surjective,  see
  Lemma~\ref{lem:weight}.
 For example  this occurs
 if $X$ is irreducible and geometrically unibranch.

 For any ring $A$, the rank one $A$-linear local systems on $X(\CC )$ are in
bijection with $ \Hom(\Pi,A^\times)$. The functor
\[
B\mapsto  \Hom(\Pi,B^\times)
\]
on the category of $A$-algebras  is corepresentable by the group ring $A[\Pi]$, so
 the {\em character variety} 
$$\Mg^\Pi_A=\Spec (A[\Pi])$$ is the fine moduli space of rank one $A$-linear local systems on $X(\CC
)$ and  is a multiplicative affine group scheme over $A$.

In this note we study closed subsets of $\Mg^\Pi_{\QQp}$ which are integral and stabilized
by an arithmetic Galois group in a suitable sense. In fact we show that those subsets are
necessarily of a very special form, namely they consist of torsion translated motivic
subtori. By a {\em motivic subtorus} of $\Mg^\Pi_A$ we mean a subgroup scheme,
 which roughly speaking is the character variety of a $1$-motive. More
  precisely, it is a subgroup scheme of the form
$\Mg^{\Pi'}_A \hookrightarrow\Mg^\Pi_A $, where $\Pi'$ is a torsion free quotient of $\Pi$
such that $\Pi'_\QQ$ underlies a quotient mixed Hodge structure of $\Pi_\QQ$, see
Proposition~\ref{prop:mot} for equivalent definitions.

 Our work is motivated by Simpson's classical article \cite{Sim93}, in which he
studies  closed algebraic subsets of $\Mg^\Pi_{\CC}$   which remain algebraic after applying the
Riemann-Hilbert correspondence. One of his central
results reads as follows.
When $X$ is smooth projective, there is  also a fine moduli space ${\rm Pic}^\nabla(X)$ of rank one integrable connections on $X$. The Riemann-Hilbert correspondence yields an  isomorphism of complex Lie groups
$$ {\rm RH}:  \Mg^\Pi_\CC (\CC)\simeq {\rm Pic}^\nabla(X)(\CC)$$
 which is not an algebraic isomorphism if the dimension is at least one.
In \cite[Thm.~3.1 (c)]{Sim93} he  proves that if $S \subset \Mg^\Pi_\CC$ is a
Zariski closed irreducible  subset such that ${\rm RH} (S(\CC))$ is again Zariski closed,
then $S$  is a motivic subtorus translated by a  character $\chi$.  When  $F=\overline{\QQ
}\subset \CC$, $S$ and ${\rm RH}(S)$ are defined over $F$, then $\chi$ can be chosen to be
a  torsion character  \cite[Thm.3.3]{Sim93}.

\smallskip

In order to formulate our main result, Theorem~\ref{thm:main} below, we
  have to introduce a Galois action related to the character variety.
Let $\pi$ be the profinite completion of $\Pi$.

\begin{ass}
From now on we assume that $F$ is finitely generated. 
\end{ass}

 Let $G={\rm Aut}(\overline F/F)$ be its Galois group.
As $\pi$ is equal to $\pi_1^{\et, \rm ab}( X_{\overline F})$,  $G$ acts on $\pi$.
 We consider  the group  ${\rm Hom}_{\rm cont} ( \pi, \overline \QQ_p^\times)$ of continuous homomorphisms, 
which  by  definition  is the colimit of $\Hom_\cont (\pi ,E^\times)$ over all finite
extensions $E\supset \QQ_p $ inside $\QQp$. The homomorphism $\Pi\to \pi$ induces
a composite map
\ml{}{ \varphi : {\rm Hom}_{\rm cont} ( \pi, \overline  \QQ^\times_p) ={\rm Hom}_{\rm cont} ( \pi, \overline {\ZZ}^\times_p)
= {\rm Hom}( \Pi, \overline \ZZ^\times_p)\\  \hookrightarrow
 {\rm Hom}( \Pi, \overline \QQ^\times_p)=\Mg^\Pi_{ \overline \QQ_p} ( \QQp) \to \Mg^\Pi_{\QQp},
 \notag}
where $\overline \ZZ_p$ is the ring of integers of $\overline \QQ_p$.   Our
main theorem is a $p$-adic arithmetic equivalent of Simpson's result recalled above,   which is expressed using the map
$\varphi$.

\begin{theorem} \label{thm:main} Assume that the weights of $\Pi_\QQ$ are in $\{-2,-1\}$.
Let $S\subset \Mg^\Pi_{\overline \QQ_p}$ be a   Zariski closed subset such that $\varphi^{-1}(S)$ is
stabilized by an open subgroup of $G$.
Then
\begin{itemize}
\item[1)]
every irreducible component $S'$ of $S$ such that $\varphi^{-1}(S')$ is non-empty  is a subtorus  $T$  of  $ \Mg^\Pi_{\overline \QQ_p}$
translated by a torsion character;

\item[2)]
  each such $T$ is a motivic subtorus.
\end{itemize}
\end{theorem}
 {
  The condition on weights in   Theorem~\ref{thm:main}   is optimal in the sense that the
  weight zero part of $\Pi_\QQ$ describes   the homology of the dual complex of 
 the
  configuration of irreducible components of $X$.  For example if $X$ is a rational double point curve defined over $\QQ$,  with normalization $X'\to X$ such that the two points above the node are still rational, then  ${\rm Hom}_{\rm cont} ( \pi, 
  \QQpp^\times) =\overline \ZZ_p^\times$ on which $G$ acts as the identity, and 
  \[\varphi: \overline \ZZ_p^\times \to \Mg^\Pi_{ \overline \QQ_p}= \mathbb{G}_{m, \overline \QQ_p}.\] So taking $S$ to be the image by $\varphi$ of a 
    non-torsion element  in $
  \overline \ZZ_p^\times $  yields a counter-example. 

Note that  by the characterization of motivic subtori in
  Proposition~\ref{prop:mot}, the converse of Theorem~\ref{thm:main} also holds, i.e.\ for a closed subset
$S\subset \Mg^\Pi_{\QQp}$ satisfying 1) and 2) of the theorem there exists an open
subgroup   of $G$ stabilizing $\varphi^{-1}(S)$.

\medskip

Theorem~\ref{thm:main} immediately implies the following independence of $p$ result. Let
$p'$ be a prime number and let
$\iota:\QQp \xrightarrow{\simeq} \QQpp$ be an isomorphism of fields. It induces an
isomorphism of schemes $\iota:\Mg^\Pi_{\QQp} \xrightarrow{\simeq}\Mg^\Pi_{\QQpp},$  so in particular, if $S
 \hookrightarrow \Mg^\Pi_{\QQp}$ is a closed subset, $\iota(S)  \hookrightarrow \Mg^\Pi_{\QQpp}$ is a closed subset as well.
In the same way as
$\varphi$ one defines 
\[
  \varphi': {\rm Hom}_{\rm cont} ( \pi,
  \QQpp^\times) \to \Mg^\Pi_{\QQpp}.
\]

\begin{corollary}\label{cor:indp}
  With the assumptions as in Theorem~\ref{thm:main},
let $S\subset \Mg^\Pi_{\QQp}$ be a closed subset such that $\varphi^{-1}(S)$
is  stabilized by an open subgroup of $G$ and let $S'$ be an irreducible component of $S$
such that $\varphi^{-1}(S')$ is non-empty. Then $(\varphi')^{-1}(\iota(S'))$
is non-empty and stabilized
by an open subgroup of $G$.
\end{corollary}

We do not even know conjecturally what a complete analog of Theorem~\ref{thm:main} could be for
local systems of rank greater one. However, we formulate a conjectural analog of
Corollary~\ref{cor:indp} for higher rank in Section~\ref{sec:rmks}.

\smallskip

We now give a summary of the methods used  in the proof of Theorem~\ref{thm:main}.  For
simplicity, we assume that $X$ is smooth and  proper over $F$
 and that $S$ is irreducible with $\varphi^{-1}(S)\ne \varnothing$.
We first discuss 1), which is proved in Section~\ref{sec:global}.  Bogomolov  in \cite{Bog80} shows
 the existence of an element of  $\sigma \in G$ which acts as a homothety of infinite
 order on the pro-$p$-completion $\pi^{(p)}$ of $\pi$.  We first use this element to show
 that torsion points are  Zariski dense on $S$. To this aim, we show that $\varphi^{-1}(S)$ is
   `$p$-adically conical' around  every torsion point of $ {\rm Hom}_{\rm cont} (\pi, \QQp^\times)$, see Proposition~\ref{prop:globaltorsion}.

 Using again $\sigma$, we further show, using $p$-adic analysis, that
 $\varphi^{-1}(S)$ is linear around smooth torsion points.   By a classical argument, this implies
 that $S$ is
 a subtorus translated by a torsion point,  see Section~\ref{sec:p-global}.  In order to
 prove 2) in Section~\ref{sec:ab}, i.e.\ motivicity of this subtorus,  we use Faltings' theorem (Tate conjecture) on homomorphisms of
 abelian varieties over $F$ applied to the Albanese variety  of $X$.

 We give an alternative shorter proof of a  weaker version of
 Theorem~\ref{thm:main} in Section~\ref{sec.alter} in which $\Mg^\Pi_{\QQp}$ is replaced by 
 $\Mg^\Pi_{K}$, where $K$ is a number field.
 It is based on known cases of
 de Jong's conjecture and the Mordell--Lang conjecture for tori.

 \medskip

 One of the main applications of Simpsons' result or of Theorem~\ref{thm:main} is to
 cohomological jumping loci.
It is a classical question, initiated by  \cite{GL87} for cohomology of coherent sheaves,   how for a fixed bounded constructible complex
$\sF\in D^b_c(X(\CC),\CC )$ the dimensions
\[
h^i( \sF \otimes L) = \dim H^i(X(\CC) ,\sF\otimes L)
\]
vary with $L\in \Mg^\Pi_\CC(\CC)$.
It is not hard to see that for all $j\in \ZZ$ the jumping loci
\[
\Sigma^i(\sF,j) = \{L \in  \Mg^\Pi_\CC(\CC) \ |\ h^i(\sF\otimes L) >j \}
\]
are Zariski closed in $\Mg^\Pi_\CC(\CC)$, see Section~\ref{sec:jump}.

\smallskip

Recall that if  $K$ is a number field and $\sF \in D^b_c(X(\CC), K)$, then for all but  finitely many  places
$K\hookrightarrow \overline \QQ_p$  of $K$,
 $\sF$ induces an object $\sF_\et$ in $D^b_c(X_{\CC,\et}, \overline \QQ_p)$.

\begin{definition} \label{defn:arithm}
 For a prime number $p$ we say that
$\sF_\et\in D^b_c(X_{\CC, \et},\overline \QQ_p )$ is {\it arithmetic} if there is a finitely
generated field extension $F'\supset F$
such that for all $\sigma\in \mathrm{Aut}(\CC /F')$ the
complex $\sigma(\sF_\et)$ is isomorphic to $\sF_\et$ inside $D^b_c(X_{\CC, \et},\overline \QQ_p)$.

We say that $\sF\in D^b_c(X(\CC),\CC )$ is {\it arithmetic} if there exists a number field $K$
such that $\sF$ descends to an object of $ D^b_c(X(\CC),K)$ and such that for
 infinitely many embeddings  $K\hookrightarrow \overline \QQ_p$,
$\sF$ induces an  arithmetic object $\sF_\et$ in $D^b_c(X_{\CC,\et}, \overline  \QQ_p)$.

\end{definition}

\begin{remark}
Note that any perverse semi-simple object $\sF\in D^b_c(X(\CC),\CC )$ which is of
geometric origin in the sense of~\cite[6.2.4-6.2.5]{BBD82} is arithmetic,  in fact, $\sF_\et \in D^b_c(X_{\CC,\et}, \overline  \QQ_p)$
is then arithmetic for all but finitely many embeddings  $K\hookrightarrow \overline \QQ_p$.
  \end{remark}

As a direct application of Theorem~\ref{thm:main} we derive in Section~\ref{sec:jump} the following corollary.

\begin{theorem} \label{thm:jump} 
  Assume that the weights of $\Pi_\QQ$ are in $\{-2,-1\}$ and that $\sF\in D^b_c(X(\CC),\CC )$ is  arithmetic.
 For $i,j\in \ZZ$, each irreducible component of
$\Sigma^i(\sF,j) $  is a motivic subtorus $T\subset \Mg^\Pi_\CC$   translated by a torsion
character $\chi$. 
  \end{theorem}

 For the same reason as in Theorem~\ref{thm:main},  the condition on the weights in Theorem~\ref{thm:jump} is sharp.
  Even for a constant sheaf $\sF$, Theorem~\ref{thm:jump} was not known in this generality.
For such $\sF$ it was known for $X$ smooth projective by Simpson \cite[Sections~6-7]{Sim93} and
for $X$ smooth by Budur--Wang  \cite[Thm.~1.3.1]{BW17}.
We refer to   Schnell~\cite[Thm.~2.2]{Schn15} for a detailed discussion of
the problem in terms of $D$-modules on
abelian varieties. All these works rely on complex analysis and the Riemann--Hilbert
correspondence.
 There is also an  approach to a
coherent version of the theorem~\cite[Prop.~1.2]{PR04} using mod $p$ reduction of the variety.

\medskip
\noindent
{\it Acknowledgements:}  We thank T.~Saito who asked us about the behaviour of our
Theorem~\ref{thm:main} under the change of complex embedding $F\hookrightarrow \CC$, see
Remark~\ref{rmk:tak} for an answer, and V.~Srinivas for a discussion on morphisms from a
smooth variety to an abelian variety which do not extend to a normal compactification.
We thank N.~Budur and B.~Wang for discussions about their results on jumping loci.
We thank M.~Groechenig  who suggested to us an alternative to our argument in
  Section~\ref{sec:p-global}  based on the Mordell--Lang conjecture for
  tori, see Remark~\ref{rmk.morlan}. We are grateful for  several helpful improvements
  which the referee suggested.

  \section{Some $p$-adic analysis} \label{sec:ana}

  In our local $p$-adic arguments below we need some $p$-adic analysis
  summarized in this section.
  In the sequel, $\QQp$ is an algebraic closure of $\QQ_p$ and $\overline\ZZ_p$
  is its ring of integers.  We endow $\QQp$ with the $p$-adic absolute value
  $\vert \cdot \rvert:\overline \QQ_p\to \RR$ normalized
  by $\lvert p \rvert = 1/p$ and $\QQp^d$ with the norm $\lvert (x_1,\ldots , x_d) \rvert= \sup_i \lvert x_i \rvert $.

  For an integer $d\ge 0$ and for  $\rho \in \lvert \QQp^\times \rvert$,  we define the {\it closed
    polydisc} $$P^d (\rho)=\{x\in \QQp^d, |x|\le \rho\}$$ around $0$ of radius $\rho$ and we let
  $$P^d(\rho; y)=P^d (\rho)+y$$ be the corresponding polydisc around $y\in \QQp^d$.
   We
  get  an isomorphism $$P^d(\rho;y)=\prod_{i=1}^d P^1(\rho;y_i) \ {\rm  with}  \ y=(y_1,\ldots, y_d).$$

\begin{definition} \label{defn:rigid}
A function $f:P^d(\rho;y)\to E$ is called
{\it globally analytic} (or {\it  rigid analytic})
  if it can be expressed as a convergent power
series with coefficients in a finite extension $E\subset \QQp$ of $\QQ_p$, i.e.\ in
coordinates as above there exist $a_J\in E$ for $J\in (\ZZ_{\ge 0})^d$
such that
$$f(x_1,\ldots, x_d) =\sum_{J=(j_1,\ldots, j_d)}  a_J (x-y)^J, \ \ (x-y)^J:=
(x_1-y_1)^{j_1}\cdots (x_d -y_d)^{j_d}, $$ converges for all $x\in P^d (\rho;y)$.
\end{definition}
The convergence
of the power series in Definition~\ref{defn:rigid} means that
$\lvert a_J\rvert\ \rho^{\lvert J\rvert} \to 0$ as $|J| \to \infty$,
so the $\QQp$-algebra of
globally analytic functions $A(P^d (\rho; y))$ on $P^d (\rho; y)$ is  isomorphic to
colimit of Tate algebras
\[\QQp \langle T_1,\ldots, T_d \rangle = \colim_{E} E\langle T_1,\ldots, T_d \rangle \]
where $E$ runs through all finite extensions $E$ of $\QQ_p$ inside $\QQp$ and  where we set
$T_i=(x_i-y_i)/z$ with $z\in \QQp$, $\lvert z \rvert =\rho$.
In the following proposition we summarize some properties of this Tate algebra over $\QQp$,
which are proved in  the same way as over a complete $p$-adic field $E$,
see~\cite[Ch.~3]{FvdP04}.

\begin{proposition}\label{prop:tatealg} We have:
  \begin{itemize}
  \item[1)] $\QQp \langle T_1,\ldots, T_d \rangle$ is a noetherian ring of dimension $d$;
    \item[ 2)] all maximal ideals of $\QQp \langle T_1,\ldots, T_d \rangle$ have residue
      field $\QQp$;
  \item[ 3)]    $\QQp \langle T_1,\ldots, T_d \rangle$ is Jacobson, i.e. every prime ideal of $\QQp \langle T_1,\ldots, T_d \rangle$ is an
    intersection of maximal ideals;
  \item[4)]
    the canonical map $\QQp[T_1,\ldots , T_d]\to  \QQp \langle T_1,\ldots, T_d \rangle$ is flat.
  \end{itemize}

\end{proposition}
\begin{proof}
 Only the Noether property 1) is not immediately deduced from the corresponding  classical results
 about Tate algebras over complete fields, see \cite[Ch.~3]{FvdP04}.
 The standard proof  that a non-zero ideal $I\subset \QQp \langle T_1,\ldots,
 T_d \rangle$ is finitely generated  goes through.
  Weierstrass preparation produces after
applying an automorphism a $T_d$-distinguished  element $f\in I$. Then $I$ is generated by
$I\cap \QQp\langle T_1 , \dots , T_{d-1} \rangle [T_d] $ and we can argue by induction on $d$.
     \end{proof}

 \begin{definition}  \label{defn:cl_ana} A subset $S\subset P^d (\rho;y)$ is  {\em closed analytic} if it is the
zero locus of finitely many globally analytic functions on $ P^d(\rho;y)$.
\end{definition}

 A {closed analytic} $S\subset P^d (\rho;y)$ corresponds to an ideal $I\subset A(P^d
 (\rho; y))$ which equals its own nilradical.
We denote by $$A(S)= A(P^d (\rho; y))/I$$ the {\it ring of globally analytic functions on}  $S$ which by definition are
restrictions of globally analytic function on $P^d(\rho; y)$. An {\em irreducible
  component} of $S$ is by definition the zero locus of a minimal prime ideal of $A(S)$, so
$S$ is the finite union of its irreducible components by Proposition~\ref{prop:tatealg} 1).

A map $f:P^{d'} (\rho';z)\to S$ is called {\it  globally analytic}  if the $d$ composed maps
$$ P^{d'} (\rho';z)\xrightarrow{f}    S\hookrightarrow  P^d(\rho;y) \to P^1(\rho;y_i) $$
 are.
For a point $x\in S$ we define the {\it stalk of analytic functions}  $A_{S,x}$ at $x$
in the obvious way. It is a noetherian local ring,  faithfully flat over the
localization $A(S)_x$ of $A(S)$ at $x$  as the completion at $x$ is. The {\em tangent space} $T_{x} S$ of $S$ at $x$ is defined as
$\Hom_{\QQp}(\mathfrak m_{S,x}/ \mathfrak m_{S,x}^2,\QQp)$, where $\mathfrak m_{S,x}$ is the
maximal ideal of the local ring $A_{S,x}$.
A point $x\in S$ is called {\it smooth}  if the ring $A_{S,x}$ (or equivalently the ring
$A(S)_x$) is regular, see~\cite[Prop.~4.6.1]{FvdP04}.

From Proposition~\ref{prop:tatealg} 1) we immediately deduce the classical:
\begin{lemma}[Strassmann] \label{lem:str}  If $d=1$ the closed analytic subsets are finite or the whole~$P^1 (\rho;y)$.
\end{lemma}

In the following we assume that we are given a direct sum decomposition $\QQp^d=\QQp^{d_1}\oplus \QQp^{d_2}$
of $\QQp$-vector spaces. We  write $P(\rho)$   for
$P^d (\rho)$, $P_1(\rho)$ for $P^{d_1}(\rho)$ and $P_2(\rho)$ for $P^{d_2}(\rho)$. The linear projections are
denoted
\ga{pq}{  r:P(\rho) \to P_1(\rho), \  q:P(\rho) \to P_2(\rho).}
We consider the action
 \ga{sigma}{\overline \ZZ_p\times P(\rho)\to P (\rho), \  (\alpha , (x_1,x_2 ))\mapsto  \alpha \cdot (x_1,x_2)= (\alpha x_1 , \alpha^2 x_2).}

\begin{definition} \label{defn:linear} Let $\rho\in \lvert  \QQp^\times \rvert$.
\begin{itemize}
\item[1)] A  subset $S\subset P(\rho)$ is said to be {\it conic} if  $\overline \ZZ_p \cdot S=S$.
 \item[2)]
A subset $S\subset P(\rho)$ is said to be {\it  linear  }  if
there exist   $\QQp$-linear subspaces $V_1\subset
\QQp^{d_1}$ and $V_2\subset \QQp^{d_2}$ such  $S =(V_1\oplus V_2)\cap P(\rho)$. In particular linear
implies conic and closed analytic.
\end{itemize}
\end{definition}

Fix $\alpha\in \overline \ZZ_p^\times$ which is not a root of unity and consider the automorphism
$$\sigma:P(\rho)\to P(\rho), \ \sigma(x_1,x_2)=(\alpha x_1,\alpha^2 x_2). $$

\begin{lemma} \label{lem:conic}

Let $S\subset P(\rho)$ be a closed analytic subset with $\sigma(S)\subset S$. Then $S$ is conic.

\end{lemma}

 \begin{proof}  Consider a  globally analytic function $f:P(\rho)\to \QQp$ vanishing on $S$.
   For $x\in S$ we consider the globally analytic function $g:\overline \ZZ_p\to \QQp$,
   $g(\beta)=f(\beta\cdot   x)$.
    By assumption, the set $\{ \alpha^n\ | \ n\ge 0\}$ is infinite and $g$ vanishes on it.
 Lemma~\ref{lem:str}
implies that $g$ vanishes identically.
Thus $f$ vanishes on $\ZZp \cdot x$. As $S$ is the zero locus of all such $f$ by
assumption, we get $ \ZZp\cdot x\subset S$.
\end{proof}

\begin{lemma}\label{lem:locsame}
If $S,S'\subset P(\rho)$ are conic closed analytic subsets such that there exists $\rho'\in
(0,\rho)$ with $S\cap P(\rho')=S'\cap P(\rho')$ then $S=S'$.
\end{lemma}

\begin{proof}
Let $f:P(\rho)\to \QQp$ be globally analytic and assume that $f$ vanishes on $S'$. We want to
show that $f$ vanishes on $S$. For $x\in S$  we consider the globally analytic function   $g:\ZZp\to \QQp$,
   $g(\beta)=f(\beta\cdot   x)$. Then for $\lvert \beta \rvert < \rho'/\rho$ we have
   $g(\beta)=0$. So $g$ has infinitely many zeros and must vanish identically by
   Lemma~\ref{lem:str}. In particular $0=g(1)=f(x)$.
\end{proof}

 \begin{proposition} \label{prop:lin}
 Consider a  closed analytic subset  $S\subset P(\rho)$ such that $\sigma(S)\subset S$ and such
 that $S$ is smooth at the point
 $0$. Assume that there exists a closed analytic subset $S_2\subset P_2(\rho)$  such that $q(S)\subset S_2$ and  such that $d_0q:T_0(S)\to
 T_0(S_2)$ is surjective.  Then $S$  is linear.
 \end{proposition}

\begin{proof}
 In order to simplify the notation we canonically  identify the tangent space $T_0(P(\rho))$ with $\QQp^d$.
   From Lemma~\ref{lem:conic} we deduce that $S$ is conic.
By the eigenspace decomposition of $d_0\sigma:T_0S \to T_0S $ we see that
$T_0(S)=V_1\oplus V_2$ with $V_1= \QQp^{d_1}\cap T_0(S)$ and  $V_2= \QQp^{d_2}\cap T_0(S)$.

It is sufficient to show that
\eq{eq.compconic}{
 S  \subset T_0(S) \cap P(\rho).
}
Indeed, by dimension reasons this inclusion has to induce an isomorphism on the stalks of
analytic functions around $0$. From  this we conclude that there exists $\rho'\in (0,\rho]$ such that $S\cap
P(\rho')=T_0(S)\cap P(\rho')$, so the proposition follows from Lemma~\ref{lem:locsame}.

In order to show~\eqref{eq.compconic} it is sufficient to show $r(x)\in V_1$ and
$q(x)\in V_2$ for any $x\in S$.  Fixing $x\in S$, by Lemma~\ref{lem:conic} we have an analytic map  $g_x:\ZZp\to
S$, $\beta\mapsto \beta \cdot x$ with $\im (d_0g_x)= \QQp\cdot (r(x) \oplus 0)$. So we obtain $\QQp\cdot
(r(x) \oplus 0)\subset T_0(S)=V_1\oplus V_2$, thus $r(x)\in V_1$.

By definition and Lemma~\ref{lem:conic}, we have
$q(\ZZp\cdot x)\subset  q(S)\subset S_2$. Using Lemma~\ref{lem:str} one deduces that
$\beta q(x)\in S_2$ for all $\beta\in \ZZp$.  So we get an analytic map $h_{q(x)}:\ZZp\to S_2 $,  $\beta\mapsto \beta
q(x)$ with $\im(d_0 h_{q(x)})=\QQp q(x)$. Here the action of $\ZZp$ on $V_2$ is just the
(unweighted) linear action.
So  we obtain $\QQp q(x)\subset T_0(S_2)$. By the surjectivity assumption, $T_0(S_2)=V_2$.  Thus $q(x)\in V_2$. This finishes the proof.  \end{proof}

  \begin{remark} \label{rmk:weight}
The results of this section carry over to the case where the power $2$ in the definition
of a conic set etc.~is replaced by any positive natural number.
\end{remark}

\section{The $p$-adic exponential map}\label{sec:exp}

In this section we recall some properties of the $p$-adic exponential and
logarithm functions in our setting, see~\cite[Ch.~12]{Cas86}
 for proofs.

For a topologically finitely generated pro-finite abelian group
 $\pi$ and a prime number $p$ we denote by $\pi^{(p)}$ its pro-$p$ completion and
 by $\pi^{(p)}/{\rm tor}$ the quotient modulo the torsion subgroup.  Let
 $\pi^{(p)}_{\QQ_p}$ be the finite dimensional $\QQ_p$-vector space $\pi^{(p)}\otimes_{\ZZ_p} \QQ_p$.

The set of {\it continuous
characters} $$H^\pi_p=\Hom_\cont (\pi^{(p)}/{\rm tor} , \QQp^\times),$$
 is after the choice of  coordinates
 $\pi^{(p)}/{\rm tor}=\oplus_1^d \ZZ_p\cdot \gamma_i$,
 isomorphic to the
open polydisc
$$\{ x\in \QQp^d \ | \ \lvert x\rvert <1 \},$$
by sending $\chi\in H^\pi_p$ to $(\chi( \gamma_i)-1)_{i=1}^d$ in  $\QQp^d$.

For $\rho\in \lvert \QQp^\times\rvert$  in the open interval $(0,1),$ we denote by $H^\pi_p( \rho )$ the subgroup of characters $\chi
\in H^\pi_p$ with $\lvert \chi(\gamma)-1 \rvert\le \rho$ for all $\gamma\in \pi$.  Clearly,
in coordinates $H^\pi_p(\rho)$ is identified with the closed polydisc $P^d(\rho;1)$.

Similarly, we consider
$$T^\pi_p=  \Hom_\cont (\pi , \QQp )  $$
and the subset $T^\pi_p(\rho)  \subset  T_p^\pi$ consisting of those $\chi\in T^\pi_\rho$ with $\lvert \chi(\gamma
) \rvert \le \rho$ for all $\gamma \in \pi$. In coordinates $T^\pi_p(\rho)$ is identified with the closed polydisc $P^d(\rho)$
 by sending $\chi\in T^\pi_p (\rho)$ to the image $(\chi( \gamma_i))_{i=1}^d$ in  $\QQp^d$.
There is a canonical isomorphism between $T^\pi_p$ and the {\it analytic tangent space} $T_1 H^\pi_p(\rho)$ at the
character $1$ defined above.

Set $\rho_\circ= p^{-\frac{1}{p-1}} $. For $\rho\in (0,\rho_\circ )\cap \lvert  \QQp^\times\rvert$
the $p$-{\em adic exponential homomorphism} is  a globally analytic
isomorphism
\[
\exp: T^\pi_p(\rho) \xrightarrow{\cong} H^\pi_p(\rho)
\]
This exponential
map is explicitly given by
\[
  \chi \mapsto \sum_{j=0}^\infty \frac{\chi^j}{j!}
\]
and its inverse
\[
\log: H^\pi_p(\rho)   \xrightarrow{\cong} T^\pi_p(\rho)
\]
it the usual globally analytic $p$-{\it adic logarithm}.

\smallskip

\begin{data} \label{data:data0} For the rest of this section we assume given the following data:
\begin{itemize}
  \item
    a closed subgroup $ \pi_\natural \subset \pi$;
  \item a continuous automorphism $\sigma: \pi  \xrightarrow{\simeq}    \pi$
which  induces a $\QQ_p$-linear  semi-simple  map  on     $\pi^{(p)}_{\QQ_p} $
stabilizing $ \pi_{\natural,\QQ_p}^{(p)}$,
   which is multiplication by $\alpha^{-2}$ on  $
   \pi_{\natural,\QQ_p}^{(p)}$ and by $\alpha^{-1}$ on  $ (\pi/\pi_\natural)^{(p)}_{\QQ_p}$,   where  $\alpha\in \ZZ_p^\times$   is not a root
of unity.

\end{itemize}
\end{data}

We denote the canonical projection $H^\pi_p\to H^{\pi_\natural}_p $ by $q$  defined by the restriction of characters.
We choose a $\QQ_p$-basis $e'_1,\ldots , e'_{d_1+d_2}$ of $\pi^{(p)}_{\QQ_p}$ consisting
of eigenvectors of $\sigma$ with $\sigma(e'_i)=\alpha^{-1}e'_i$ for $i\le d_1$ and with  $\sigma(e'_i)=\alpha^{-2}e'_i$ for $i> d_1$.
We can furthermore choose the eigenvectors such that
\[
  \pi^{(p)}/{\rm tor} \subset \pi' :=\ZZ_p e'_1 + \cdots + \ZZ_p
  e'_{d_1+d_2} .
\]
We get an isomorphism $T^\pi_p\cong \QQp^{d_1}\oplus \QQp^{d_2}$. Choose an integer $w\ge 0$
 such that $p^w \pi'\subset  \pi^{(p)}/{\rm tor} $.
 On $T^{\pi'}_p$ the linear map  $\sigma$ acts as
 in Section~\ref{sec:ana} and there are $\sigma$-equivariant inclusions $
 T^\pi_p(\rho/p^w) \subset T^{\pi'}_p(\rho) \subset  T^\pi_p(\rho)$.   So via the exponential isomorphism we can deduce from
 Lemma~\ref{lem:conic} and Proposition~\ref{prop:lin} the following result.

 \begin{proposition}\label{prop:Hpi}
    For some  $\rho\in (0,\rho_\circ)\cap \lvert \QQp^\times\rvert$  let $S\subset
    H^\pi_p(\rho)$ be a  closed analytic subset which is
   stabilized by $\sigma$ such that $S\cap  H^\pi_p(\rho/p^w) $ is non-empty
   \begin{itemize}
   \item[1)] Then the character $1$ is in $ S$.
     \item[ 2)] If moreover $1$ is a smooth point of $S$ and if there exists a closed
       analytic subset $S_\natural\subset H^{\pi_\natural}_p(\rho)$ such that $q(S)\subset
       S_\natural$ and such that $d_1 q:T_1 S \to T_1 S_\natural$ is surjective, then
       $S\cap H^\pi_p(\rho/p^w) $ is a subgroup of $H^\pi_p(\rho/p^w) $.
   \end{itemize}

   \end{proposition}

\section {Torsion points in   character varieties}\label{sec:torsion}

We start with recalling some simple facts about global character varieties.
Let $\Pi$ be a finitely generated abelian group and let
$A$ be commutative ring.  One defines the functor $\{A-\text{algebras} \}\to  {\rm Ab
}$ to the category of abelian groups which assigns to $B$ the set of characters $\Pi\to
B^\times$.  It is corepresented  by the group ring  $A[\Pi]$.
So we define the {\it character variety} as the commutative group scheme
$$\Mg^\Pi_A=\Spec(A[\Pi]).$$

Recall that $\pi^{(p)}$ denotes the pro-$p$ completion of a pro-finite group $\pi$.

\begin{data} \label{data:data} In this section and the next section we assume that we are given the following data:
\begin{itemize}
  \item
    an inclusion of finitely generated abelian groups $\Pi_\natural\subset \Pi$, the
    pro-finite completion of which  we denote by  $ \pi_\natural \subset \pi$;
  \item a prime number $p$;
  \item a continuous automorphism $\sigma: \pi  \xrightarrow{\simeq}    \pi$
which  induces a $\QQ_p$-linear  semi-simple  map  on     $\pi^{(p)}_{\QQ_p} $
stabilizing $ \pi_{\natural,\QQ_p}^{(p)}$,
   which is multiplication by $\alpha^{-2}$ on  $
   \pi_{\natural,\QQ_p}^{(p)}$ and by $\alpha^{-1}$ on  $ (\pi/\pi_\natural)^{(p)}_{\QQ_p}$,   where  $\alpha\in \ZZ_p^\times$   is not a root
of unity.

\end{itemize}
\end{data}

We denote the projection $\Mg^{\Pi}_{\QQp}\to \Mg^{\Pi_\natural}_{\QQp}$ by $q$.

For a character $\chi\in \Hom_\cont (\pi,\overline \QQ_p^\times ) = \Hom_\cont (\pi,\overline \ZZ_p^\times ),$ we let
$$\bar\chi \in \Hom (\pi,\overline \FF_p^\times )$$
 denote its associated {\it  residue character}. For a character
$\xi \in \Hom (\pi,\overline \FF_p^\times )$ we let $$[\xi]\in \Hom_\cont (\pi,\overline
\QQ_p^\times )$$ denote its {\it Teichm\"uller lift}.  It is defined by post-composing
$\xi$ with the Teichm\"uller lift $\overline \FF_p^\times \hookrightarrow W(\overline
\FF_p)^\times$ followed by the inclusion $W(\overline \FF_p)^\times  \hookrightarrow \overline \QQ_p^\times$.

We consider the composite map
\ml{}{ \varphi: {\rm Hom}_{\rm cont} ( \pi, \overline \QQ_p^\times) = {\rm Hom}_{\rm cont} ( \pi, \overline \ZZ_p^\times)= {\rm Hom}( \Pi,  \overline \ZZ_p^\times) =\Mg^\Pi_{ \overline \ZZ_p }(\overline \ZZ_p )\\
\hookrightarrow{}  {\rm Hom}( \Pi,  \overline \QQ_p^\times)
=
\Mg^\Pi_{ \overline \QQ_p}( \overline \QQ_p) \to \Mg^\Pi_{ \overline \QQ_p }.\notag
}
Note the homomorphisms $\pi\to \pi^{(p)}\to \pi^{(p)}/{\rm tor}$ induce an injective homomorphism
$$H^{\pi}_p \hookrightarrow
{\rm Hom}_{\rm cont} ( \pi, \overline  \QQ_p^\times).$$
In  the following lemma we collect properties of $\varphi$ that we need.

\begin{lemma}\label{lem:irrcomp}
  Let  $\rho\in \lvert \QQp^\times\rvert$  be in the interval $(0,1)$ and let $S\subset \Mg^\Pi_{ \overline \QQ_p}$ be a
 Zariski closed subset.
  \begin{itemize}
    \item[1)]
The preimage $D=\varphi^{-1}(S)\cap H^{\pi}_p(\rho)$
 is a closed
 analytic subset of $H^{\pi}_p(\rho)$ and there is an identification of tangent spaces
 $T_1 D \xrightarrow{\simeq} T_1 S$  if $1$ is in $S$.
\item[ 2)]  If $\sigma$ stabilizes
  $\varphi^{-1}(S)\cap H^\pi_p(\rho)$ then there exists an integer $m>0$ such that
  $\sigma^m$ stabilizes  $\varphi^{-1}(S')\cap H^\pi_p(\rho)$ for each irreducible
  component $S'$ of $S$.
  \item[ 3)] If $S$ is irreducible and if $D=\varphi^{-1}(S) \cap H^{\pi}_p(\rho)$ is
    non-empty then $\varphi(D)$ is dense in $S$.
    \item[ 4)]  If $S$ is irreducible and if $\varphi^{-1}(S) \cap H^{\pi}_p(\rho)$ is a
      subgroup of $H^{\pi}_p(\rho)$ then $S$ is a subtorus of $ \Mg^\Pi_{ \overline \QQ_p }$.
\end{itemize}
\end{lemma}

\begin{proof}
By Proposition~\ref{prop:tatealg}~4)  the map $\varphi|_{H^\pi_p(\rho)}$ corresponds to a flat ring  homomorphism   $\varphi^*:\sO(\Mg^\Pi_{ \overline \QQ_p}) \to
A(H^\pi_p(\rho))$ from the algebraic functions on $\Mg^\Pi_{\QQp} $ to the globally analytic functions
on $H^\pi_p(\rho)$.  This ring homomorphism $\varphi^*$ induces an
  isomorphism on the
completions of local rings
$$\sO(\Mg^\Pi_{ \overline \QQ_p})_{(\varphi^*)^{-1}(\mathfrak m)}^\wedge \xrightarrow{\simeq}
A(H^\pi_p(\rho))_{\mathfrak m}^\wedge$$ for any maximal ideal $\mathfrak m\subset
A(H^\pi_p(\rho))$. In particular, it induces an isomorphism of tangent spaces as in~1).  
If $S$ is the zero set of the finitely many algebraic funtions $f_i\subset \sO(\Mg^\Pi_{ \overline
  \QQ_p})$, then $\varphi^{-1}(S)\cap H^\pi_p(\rho)$ is the zero set of the analytic functions
$\varphi^*(f_i)$ and therefore defines a closed analytic subset.  This proves 1).

The closed subset $S$ is defined by an ideal $I$ which is the intersection of
the  finitely many minimal prime ideals $\frak{p}_i\supset I$  defining its irreducible
components. Any prime ideal $\mathfrak q\subset A(H^\pi_p(\rho))$ which is minimal
containing $\varphi^*(\mathfrak p_i )$ is also minimal over $\varphi^*(I)$. This follows
from going-down for the flat map $\varphi^*$.
 Expressed geometrically this means that for  an irreducible component $S'$ of $S$,
  $\varphi^{-1}(S')\cap H^\pi_p(\rho)$ consists of a finite union
of irreducible components of  $\varphi^{-1}(S)\cap H^\pi_p(\rho)$. As $\sigma$ is an
isomorphism, it permutes the finitely many irreducible components of
$\varphi^{-1}(S)\cap H^\pi_p(\rho)$.
So in 2) we can
choose $m>0$ such that $\sigma^m$ stabilizes all these irreducible components.  This
finishes the proof of~2).

The conditions of 3) imply that $I =(\varphi^*)^{-1}(\varphi^*(I) A(H^\pi_p(\rho)) )$. The closure of
$\varphi(D)$ is defined by the ideal
\eq{eq.intseclem}{
  \cap_{\mathfrak m} (\varphi^*)^{-1}(\mathfrak m )
  =  (\varphi^*)^{-1} (\cap_{\mathfrak m} \mathfrak m )
}
where $\mathfrak m$ runs through the maximal ideals of $A(H^\pi_p(\rho))$ containing
$\varphi^*(I)$. By Proposition~\ref{prop:tatealg} 4) we see that $\cap_{\mathfrak m} \mathfrak m $ is
the nilradical of the ideal of $A(H^\pi_p(\rho)) $ generated by $\varphi^*(I)$, so the ideal in the equation 
\eqref{eq.intseclem} is equal to $I$, in
other words the closure of $\varphi(D)$ is $S$. This proves~3).

In order to show 4) it is sufficient to show that $S(\QQp )$ is a subgroup of
$\Mg^\Pi_{\QQp}( \QQp )$. Consider the commutative diagram
\[
  \xymatrix{
    D\times D \ar[d]_{\varphi\times \varphi} \ar[r] & D \ar[d]^{\varphi} \\
      S(\QQp) \times S(\QQp )  \ar[r] & \Mg^\Pi_{\QQp}( \QQp )
  }
\]
in which the horizontal maps are the group operations   $(a,b) \mapsto a-b$ and
where $D=\varphi^{-1}(S)\cap H^\pi_p(\rho)$. The image of the left vertical map is Zariski
dense by part~(iii). The image of the right vertical map is contained in $S(\QQp)$, so
the same
is true for the image of the lower horizontal map.
\end{proof}

For a character $\xi \in \Hom_\cont(\pi , \overline\FF_p^\times),$ we let $\Hom_\cont(\pi ,
\overline\QQ_p^\times)^\xi$ denote the set of characters $\chi \in \Hom_\cont(\pi ,
\overline\QQ_p^\times)$ with $\bar \chi = \xi$.

\begin{proposition}\label{prop:globaltorsion}
  Let $\xi \in \Hom_\cont(\pi , \overline\FF_p^\times) $ be fixed by $\sigma$.
 Let   $S\subset \Mg^\Pi_{\overline \QQ_p}$ be a  closed subset such that
 $\varphi^{-1}(S)\cap \Hom_\cont(\pi ,
\overline\QQ_p^\times)^\xi$ is stabilized by $\sigma$. Let $S'\subset S$ be an irreducible
component such that $\varphi^{-1}(S')\cap \Hom_\cont(\pi ,
\overline\QQ_p^\times)^\xi$ is non-empty.  Then there exists a torsion point
$\chi\in S'$ with $\bar\chi = \xi$.
\end{proposition}

\begin{proof}
 As the Teichm\"uller lift $[\xi]$ of $\xi$ is fixed by $\sigma$,  and has the same order as $\xi$,  translating by $[\xi]$
 is $\sigma$-equivariant on $ \Hom_\cont(\pi ,\overline\QQ_p^\times)$ and preserves the
 torsion points. So
we can replace $S$ by $[\xi^{-1}]\cdot S$  in the following and thereby assume that
$\xi=1$.

For $n>0$ we consider the Zariski closed subsets $[p^n](S')\subset [p^n](S)$ of $\Mg_{\overline
  \QQ_p}^\Pi$. We note that
$ \varphi^{-1}([p^n](S)) = [p^n] ( \varphi^{-1}(S )).$
This implies that $ \varphi^{-1}([p^n](S)) \cap H^\pi_p$ is stabilized by $\sigma$.

Choose a character $\psi\in \varphi^{-1}(S')\cap H^\pi_p$.
We fix $0<\rho'<\rho_\circ= p^{-\frac{1}{p-1}}$ and choose $n$ large such that
$ [p^n](\psi)$ lies in $H^{\pi}_p(\rho'/p^w )$, where $w$ is as in Proposition~\ref{prop:Hpi}.
 By Lemma~\ref{lem:irrcomp} there exists $m>0$ such that
$\sigma^m$ stabilizes
$D=\varphi^{-1}([p^n](S'))\cap H^{\pi}_p(\rho').$
From Proposition~\ref{prop:Hpi} 1) we deduce that $ D$ contains $1$,
therefore
 $[p^n](S')$ contains $1$ and $S'$ contains a torsion point.

\end{proof}

\begin{corollary} \label{cor:globaltorsion}
  Let   $S\subset \Mg^\Pi_{\overline \QQ_p}$ be a closed subset such that
  $\varphi^{-1}(S)$  is stabilized by $\sigma$.  Then the
  torsion points are dense on each irreducible component $S'$ of $S$ for which  $\varphi^{-1} (S')$ is non-empty.
\end{corollary}

\begin{proof}
 Let $S'\subset S$ be an irreducible component of $S$ for which  $\varphi^{-1} (S')$  is
 non-empty and let $U\subset S'$ be a non-empty open subset.
The closure   $\mathcal S'$ of $S'$  in $\Mg^\Pi_{\overline \ZZ_p}$ has the
property that $\dim(\mathcal S'_{\overline \FF_p})=\dim(S')\ge 0$.
 The closure $\sV$ of
 $S^{'} \setminus U $ in  $\Mg^\Pi_{\overline \ZZ_p}$ has the property that $\dim(\sV_{\overline \FF_p})\le \dim(S')-1$. So
 we can find $\xi\in \mathcal
S'(\overline \FF_p)\setminus \sV(\overline \FF_p)$.  The point  $\xi$ is the moduli point of a character in
$\Hom_\cont(\pi,\overline \FF_p^\times)$. So there is an  $m >0$ such  that $\sigma^m$
 fixes $\xi$. We may assume without loss of generality that $m=1$.
Note that by the choice of $\xi$, all characters $\chi\in S'( \overline \ZZ_p)$ with $\bar
\chi=\xi$ automatically lie in $U( \overline \QQ_p)$.

Now $S$, $S'$ and $\xi$ satisfy the assumptions of Proposition~\ref{prop:globaltorsion} and the
$\chi$ constructed there is a torsion point on $U$.
 \end{proof}

 \begin{remark}\label{rmk.morlan} As observed by M.~Groechenig, 
Corollary~\ref{cor:globaltorsion} together with the Mordell--Lang conjecture for tori,
shown by M.~Laurent in \cite[Introduction]{Lau84} using Diophantine approximation, see Proposition~\ref{prop.morlan},
immediately implies Theorem~\ref{thm:critsubtorus}. However, in Section~\ref{sec:p-global}
we explain how to deduce the latter
purely in terms of elementary $p$-adic analysis. 
 \end{remark}

\begin{corollary}  With the notation of Propositon~\ref{prop:globaltorsion}, $\chi \in {\rm Hom}_{\rm cont} ( \pi, \overline \QQ_p^\times)  $ is fixed by a power of  $\sigma$
if and only if it is torsion.
\end{corollary}

\section{ Criterion for a closed subset of $\Mg^\Pi_{\QQp}$ to be a subtorus  }\label{sec:p-global}
 The notation is as in Section~\ref{sec:torsion}.

\begin{proposition} \label{prop:globallin}
  Let $S\subset \Mg^\Pi_{\overline \QQ_p}$ be an irreducible closed subset such that
  $\varphi^{-1}(S)\cap H^\pi_p(\rho)$
  is stabilized by $\sigma$ for some $\rho\in (0,1)\cap\lvert \QQp^\times\rvert$.  
  Assume that $1$ is a regular point of $S$ and that there exists a
  closed subset $S_\natural \subset  \Mg^{\Pi_\natural}_{\overline \QQ_p}$ such that $q(S)\subset
  S_\natural$ and such that $d_1 q: T_1 S \to T_1 S_\natural$ is surjective.
 Then $S$ is a subtorus of $\Mg^\Pi_{\overline \QQ_p}$.
\end{proposition}
\begin{proof}
  We can assume without loss of generality that $\rho<\rho_\circ$, see Section~\ref{sec:exp}.
Write $D$ for $\varphi^{-1}(S)\cap H^\pi_p(\rho)$ and $D_\natural$ for $\varphi^{-1}(S_\natural)\cap H^{\pi_\natural}_p(\rho)$.
By Lemma~\ref{lem:irrcomp}~1) we have a canonical isomorphism of tangent spaces $T_1
D\cong
T_1 S $ and $T_1 D_\natural \cong T_1 S_\natural $, so Proposition~\ref{prop:Hpi}~2)
implies that $D\cap H^\pi_p(\rho/p^w) $ is a subgroup of $H^\pi_p(\rho/p^w) $ for some
integer $w\ge 0$. We deduce from
Lemma~\ref{lem:irrcomp}~4) that $S$ is a subtorus.
\end{proof}

\begin{theorem} \label{thm:critsubtorus}
  Let $S\subset \Mg^\Pi_{\overline \QQ_p}$ be a closed subset such that $ \varphi^{-1}(S)$
  is stabilized by $\sigma$. Then each irreducible component $S'$ of $S$ with
  $\varphi^{-1}(S')$ non-empty is a subtorus $T$ translated by a torsion character.

\end{theorem}
\begin{proof}
  By Chevalley's theorem~\cite[1.8.4]{EGAIV},  the image $q(S')$ is constructible in
  $\Mg^{\Pi_\natural}_{\QQp}$, so we can find an open subset $U_\natural\subset
  \Mg^{\Pi_\natural}_{\QQp}$ such that $q(S')\cap U_\natural$ is non-empty and closed in
  $U_\natural$. We can furthermore assume that  $q(S')\cap U_\natural$ is regular when
  endowed with the reduced subscheme structure. We let $S'_\natural$ be the closure of
  $q(S')\cap U_\natural$ in $\Mg^{\Pi_\natural}_{\QQp}$.
  Let
  $U\subset S'\cap q^{-1}(U_\natural)$ be a non-empty open subset such that $q|_U:U\to
  q(S')\cap U_\natural$ is smooth, when endowed with the reduced subscheme structure.

By Corollary~\ref{cor:globaltorsion} there exists a torsion point $\chi\in U$. After
replacing $\sigma$ by some power we can assume that $\sigma$ fixes $\chi$, so translation
by $\chi$ is $\sigma$-equivariant and we can replace $S$ by $\chi^{-1}\cdot S$ etc. Now $S'$
contains $1$ and by Lemma~\ref{lem:irrcomp}~2) for some $\rho\in (0,1)$ there
exists $m>0$ such that $\sigma^m$ stabilizes $\varphi^{-1}(S')\cap H^\pi_p(\rho)$. After
replacing $\sigma$ by $\sigma^m$ our sets $S'$ and $S'_\natural$ satisfy the assumptions of
Proposition~\ref{prop:globallin}, so $S'$ is a subtorus.
\end{proof}

\section{Subspaces of moduli spaces of rank one local systems} \label{sec:global}

Let $X$ be a reduced, separated scheme of finite type over a subfield $F$ of the field of complex numbers.  We denote by $\overline F$ the algebraic closure of $F$ in $\CC$.  Let $G={\rm Aut}(\overline F/F)$ be the Galois group of $F$.
Let $\Pi$ be the abelian fundamental group  $\pi_1^{\ab}(X(\CC ))$. The group $\Pi$ is
finitely generated and its pro-finite
completion $\pi$ is isomorphic to the \'etale fundamental group $\pi_1^{\et,\ab} (X_{\overline
  F})$ on which  $G$ acts continuously.

 We define a weight filtration $W_{-2}\Pi\subset W_{-1}\Pi \subset \Pi$, 
as follows. Choose a regular, open and dense subscheme $q:U\to X$ and an open embedding $U\to \bar U$ into a proper
smooth variety $\bar U$ over $F$  such that $\bar U\setminus U$ is a simple normal
crossings divisor. Then \[W_{-2} \Pi=   q_* \ker(\pi_1^{\ab}(U(\CC )) \to
  \pi_1^{\ab}(\bar U(\CC )) )\]
and $$W_{-1}\Pi = \im (q_*).$$

\begin{lemma}\label{lem:weight}
The  weight filtration on $\Pi_\QQ$ defined in \cite{Del74} is the same as the one defined above.   It is independent of the choice of $q$ and $U\to \bar U$ integrally. 
 In particular, the weights lie in $\{-2,-1\}$   if $X$ is
irreducible and geometrically unibranch. 
\end{lemma}

\begin{proof}
Dually, the mixed Hodge structure on $H^1(X(\mathbb C), \mathbb Q)$ is defined in (\cite[Prop.8.2.2]{Del74}) by descent
$H^1(X(\mathbb C), \mathbb Q)=H^1(X_\bullet(\mathbb C), \mathbb Q)$ where $X_\bullet\to X$ is a simplicial resolution, $X_0$.  Thus one has an exact sequence 
\ml{}{{\rm Ker} (H^0(X_1(\mathbb C), \mathbb Q) \to  H^0(X_2(\mathbb C), \mathbb Q))/{\rm Im} H^0(X_\bullet(\mathbb C), \mathbb Q) \to \\
H^1(X(\mathbb C), \mathbb Q) \to H^1(X_0(\mathbb C), \mathbb Q) \notag}
of mixed Hodge structures. 
By~\cite[Thm.8.2.4]{Del74}, the weights on the left are $0$, the ones on the right are in $\{1,2\}$. In particular the weights of $H^1(X(\mathbb C), \mathbb Q)$ lie in $\{1,2\}$ if and only if the map  $H^1(X(\mathbb C), \mathbb Q) \to H^1(X_0(\mathbb C), \mathbb Q) $ coming from the desingularization $X_0\to X$ is injective, or equivalently, as the restriction  $q$ factors through $U\to X_0$ and $
H^1(X_0(\mathbb C), \mathbb Q) \to H^1(U(\mathbb C), \mathbb Q)$ is injective, if and only if
the restriction $q^*: H^1(X(\mathbb C), \mathbb Q) \to H^1(U(\mathbb C), \mathbb Q) $  is
injective.  Finally, the weight $1$ part of $H^1(X(\mathbb C), \mathbb Q)$ is then its
subquotient $q^* H^1(X(\mathbb C), \mathbb Q)  \cap H^1(\bar U(\mathbb C), \mathbb Q)
\subset H^1(U(\mathbb C), \mathbb Q)$ and the weight $2$ part is the quotient ${\rm Im}
H^1(X(\mathbb C), \mathbb Q)  \subset H^1(U(\mathbb C), \mathbb Q)/H^1(\bar U(\mathbb C),
\mathbb Q).$ This proves the first part.

Since $\pi_1(\bar U(\CC )) $, and thus $\pi_1^{\ab}(\bar U(\CC ))$, is a birational invariant among smooth varieties, to prove the independence integrally for two opens $U, V$ we may assume that  $V\subset U \subset \bar U$ so that   $\bar U\setminus V$ is a normal crossings divisor.  Then ${\rm Ker}(\pi_1^{\ab}(V(\CC )) \to \pi_1^{\ab}(\bar U(\CC ))$ is the subgroup spanned by the local monodromies around the components of $\bar U\setminus V$, which is the union of the components of $\bar U\setminus U$ with the Zariski closure in $\bar U$ of the components of $U\setminus V$. But the subgroup in  $\pi_1^{\ab}(V(\CC ))$  spanned by the latter ones is 
${\rm Ker}(\pi_1^{\ab}(V(\CC )) \to \pi_1^{\ab}( U(\CC ))$. As the map $V\to X$ factors through $U$, this shows the independency of $W_{-2}\Pi$. Furthermore, as $\pi_1^{\ab}(V(\CC )) \to \pi_1^{\ab}( U(\CC ))$ is surjective, this shows the independency of $W_{-2}\Pi$.
Finally for the last part, we quote  \cite[Lem.~0BQI]{dJStack}. This finishes the proof. 
\end{proof}

 In the following we write $\Pi_\natural$ for $W_{-2} \Pi$.
 
For the rest of this section assume that $F$ is finitely generated and that $\Pi_\QQ$ has
weights in $\{-2,-1\}$.
By~\cite[Thm.~1, Thm.~2]{Del80},
$G$ acts purely of weight $-2$ on the pro-finite completion $\pi_\natural$ of
$\Pi_\natural$ and purely of weight $-1$ on $\pi/\pi_\natural$. As a consequence of this
weight filtration and using the Hodge-Tate property for the  Galois representation on $\pi^{(p)}_{\QQ_p}$ one obtains:

\begin{proposition}[Bogomolov, Litt]\label{prop:boli}
  If $\Pi_\QQ$ has
weights in $\{-2,-1\}$ then
for any prime number $p$ and any $\alpha\in \ZZ_p^\times$ sufficiently close to $1$ there
exists an element $\sigma\in G$ which  induces a semi-simple  map  on     $\pi^{(p)}_{\QQ_p} $
   which is multiplication by $\alpha^{-2}$ on  $
   \pi_{\natural,\QQ_p}^{(p)}$ and by $\alpha^{-1}$ on  $ (\pi/\pi_\natural)^{(p)}_{\QQ_p}$.
 \end{proposition}
The proposition is shown in \cite[Cor.~1]{Bog80} for $X$ smooth and proper, and in
\cite[Lem.~2.10]{Lit18} for $X$ smooth, which implies it in general, as   $\pi^{(p)}_{\QQ_p}$ is a
$G$-equivariant
quotient of $\pi^{\et , \ab}_1(X^{\rm reg}_{\overline F})^{(p)}_{\QQ_p} $.

\begin{proof}[Proof of Theorem~\ref{thm:main} 1)]
If we choose  $\alpha \in \ZZ^\times_p$ in Proposition~\ref{prop:boli}  not to be a  root of unity, the
resulting automorphism $\sigma:\pi\xrightarrow{\simeq}\pi$ satisfies the conditions of
Data~\ref{data:data}, so part 1) of Theorem~\ref{thm:main} is a direct consequence of Theorem~\ref{thm:critsubtorus}.
\end{proof}

\section{On subgroup schemes of semi-abelian varieties} \label{sec:ab}

 Let $F\subset \CC$ be a finitely generated field. Let $\overline F$ be its algebraic closure in $\CC$, and  $G=\mathrm{Gal}(\overline F / F)$ be its Galois group.
Let $M/F$ be a semi-abelian variety. By definition,  $M$ is an extension of an abelian variety $A$ by a
torus $T$. Let $V^M=H_1(M(\CC),\QQ )$ and let
$V_p^M = H^1_\et(M_{\overline F},\QQ_p)^\vee  $ be the dual of   $p$-adic \'etale cohomology, endowed with
its canonical
$G$-action. There is a canonical  comparison  isomorphism  $$\psi^M:V^M\otimes_\QQ
\QQ_p\xrightarrow{\simeq} V_p^M.$$

\begin{proposition}\label{prop:constate}
For a $\QQ$-linear subspace $U\subset V^M$ such that $\psi^M(U_{\QQ_p})\subset V_p^M$ is stabilized by
$G$,  there exists a unique semi-abelian subvariety $N\hookrightarrow M$  defined over $F$ such that $U=V^N$.
\end{proposition}
 Before starting the proof, let us state a small lemma.

\begin{lemma}\label{lem:semisimp}
  Let $V$ be a finite dimensional vector space over a field $K$ of characteristic $0$,  and let $K\subset K'$ be a field
  extension. Let $C\subset \End_K(V)$ be a  $K$-subalgebra and let $U\subset V$ be a
  $K$-linear subspace. If there exists an idempotent $e'\in C_{K'}$ with
  $e'(V_{K'})=U_{K'}$ then there exists an  $e\in C$ with $e(V)=U$.
\end{lemma}

\begin{proof}
Using a spreading argument one reduces to a finite field extension $K\subset K'$. Then
$e=\tr_{K'/K}(e')\in C$ is the requested element.
\end{proof}

\begin{proof}[Proof of Proposition~\ref{prop:constate}]
   Uniqueness is clear as the tangent space of $N$ at the neutral element, which is
   identified with $U_\CC$, uniquely characterizes $N$.

   Recall that Faltings \cite[Thm.~1]{Fal86}
showed that $V_p^A$ is a semi-simple $G$-representation and that for semi-abelian
varieties $M$, $ M'$ we have an isomorphism
\eq{eq.faltiso}{
V_p:\Hom_k(M,M')_{\QQ_p} \xrightarrow{\simeq} \Hom_G(V_p^M, V_p^{M'})  .
}
In fact Faltings showed the latter for abelian varieties only, but it can be extended to
semi-abelian varieties  using the Mordell-Weil theorem, see \cite[Thm.~4.3]{Jan95}.

We show the existence of $N$ in three steps.

\smallskip

  {\em  Case 1}: $M=T$ torus.\\
  In this case we have $V^M=\Hom(X(T_{\overline F}),\QQ)$ , where $X(T_{\overline F})$ is the group
  of characters, so $U$ induces by duality a free abelian quotient $Y$ of  $X(T_{\overline
    F})$.  As the induced $p$-adic quotient $ X(T_{\overline F})\otimes_\ZZ \ZZ_p \to Y\otimes_\ZZ \ZZ_p$  is stabilized by $G$, so is the quotient $X(T)\to Y$.
   We set $N_{\overline F}=\Spec(\overline F [Y])$ and
  observe that the $G$-action
   descends  this torus to the requested torus $N/ \overline F$.

\smallskip

  {\em  Case 2}: $M=A$ abelian variety.\\
 Note  that the Betti realization
$$V:  \End_k(A)_{\QQ}\to  \End(V^A) $$
 identifies $\End_k(A)_\QQ$ with a
semi-simple subalgebra of $\End(V^A)$.

Falting's results recalled above imply that  there exists an idempotent $e'\in \End_k(A)_{\QQ_p}$ such
$V_p(e')(V_p^A)=\psi^A(U_{\QQ_p})$. By Lemma~\ref{lem:semisimp} there is an
$e\in \End_k(A)_{\QQ}$ with $V(e)(V^A)=U$. Write $e=f/n$ with $f\in \End_F(A)$
and $n\in \mathbb N$. We define $N$ to be
$\im(f)\subset A$.

\smallskip

  {\em  Case 3}: $M$ semi-abelian variety with torus part $T$ and abelian quotient $A$.\\
We  construct a semi-abelian  quotient variety $M'$ of $M$ such that $N$ is the
connected component of $\ker(M\to M')$.
By Case~1 the subspace $U\cap V^T$ comes from a subtorus $T^\flat$ of $T$. Replacing $M$
by $M/T^\flat$ we may assume without loss of generality that $U\cap V^T=0$, so we get an
isomorphism $\tau:U \xrightarrow{\simeq} \tau(U)\subset V^A$, which is $G$-invariant after
tensoring with $\QQ_p$.  By Case~2
there is an abelian subvariety  $ t: A''\hookrightarrow A$
 such that $V(t)(V^{A''})=\tau(U)$. In view of the isomorphism~\eqref{eq.faltiso} there exists a map $s\in
 \Hom_F(A'',M)_{\QQ_p}$ with \'etale realization $\tau^{-1}_{\QQ_p} \circ
 V_p(t)$. We obtain a commutative diagram
 \[
   \xymatrix{
 & & A'' \ar[d]^t \ar@{-->}[ld]_s \\
 T\ar[r] & M\ar[r] & A
}
\]
where the dotted arrow means that $s$ has $\QQ_p$-coefficients.
As the realization of $s$ is induced by the $\QQ$-linear map  $\tau^{-1} \circ V(t)$, so
is $s\in  \Hom_F(A'',M)_{\QQ}$.
So $M'=\mathrm{coker}(s)$ is the requested semi-abelian
quotient of $M$.

\end{proof}

Theorem~\ref{thm:main}~2) follows immediately from the following proposition. 
  Let $Y$ be a reduced, separated scheme of finite type over $\CC$ and define $\Pi=\pi_1^{\ab}(Y(\CC ))$.
\begin{proposition} \label{prop:mot}
If $\Pi_\QQ$ has weights in $\{-2,-1\}$ and $T\subset \Mg^\Pi_{\QQp}$ is a
subtorus the following are equivalent:
\begin{itemize}
\item[1)] $T$ is a motivic subtorus,
  \item[2)] there exists a morphism $\psi_T:Y^{\rm reg}\to B$ to a
   complex semi-abelian variety $B$  such that
      \ga{}{  \tau^*(T)={\rm im}\left( \psi_T^*: \Mg^{\Gamma}_{\overline \QQ_p} \to \Mg^{\Pi^\circ}_{\overline \QQ_p} \right) \subset \Mg^{\Pi^\circ}_{\overline \QQ_p}  \notag}
where $\Gamma =\pi_1^{\rm ab}(B(\CC))$ and $\tau: \Pi^\circ= \pi_1^{ \ab }(Y^{\rm
   reg}(\CC) ) \to \Pi $ is induced by the open embedding $Y^{\rm reg} \hookrightarrow Y$ of the regular locus.
\item[3)]  $Y=X\otimes_F \CC$ for a scheme $X$ of finite type over a finitely generated
  field $F\subset \CC$ such that
  $\varphi^{-1}(T)$ is stabilized by $\mathrm{Gal}(\overline F / F)$, where $\varphi$ is
  as defined in Section~\ref{sec:torsion}. 
\end{itemize}

\end{proposition}

\begin{proof}
 Applying Lemma~\ref{lem:weight} and noticing that if $U$ is regular, the weights of 
$H_1(U(\CC),\QQ)$ are in $\{-2,-1\}$, we may assume that $Y$ is regular and connected.

We prove 1) implies 2). 
A quotient Hodge structure of a polarized Hodge
 structure inherits the polarization and by~\cite[Const.~10.1.3]{Del74},  we know that polarizable integral Hodge
 structures of weight $\{-2,-1\}$ correspond to complex semi-abelian varieties.

 We consider
 Serre's Albanese morphism $ {\rm alb}: Y\to M$ to a complex semi-abelian variety, defined
 once we have chosen a rational point $x_0\in Y(\CC )$ such that ${\rm alb}(x_0)=0$,
 see~\cite[Thm.~4]{Ser58}. We know that  ${\rm alb }: H_1(Y(\CC) ,\QQ) \to
 H_1(M(\CC ),\QQ)$
 is an isomorphism~\cite[Ann.~II]{Ser58}. By 1) our torus $T$ is induced by a torsion free
  quotient of $\Pi$ compatible with the Hodge structure, so it corresponds to a quotient
  semi-abelian variety  $B$ of $M$ and we define $\psi_T$ as the composite map
  $Y\xrightarrow{\rm alb} M \to B$. 
 
 We prove 2) implies 3).  The morphism $\psi_T$ is defined over a finitely generated
 subfield $F\subset \CC$. Therefore $\varphi^{-1}(T)$ is stabilized by ${\rm
   Gal}(\overline{ F} / F)$.
 
 We prove 3) implies 1). Let $M$ be the Albanese of $Y$ as above, which also descends to $F$. The torus $T$ induces a
 $\QQ$-linear subspace $U\subset V^M\simeq H_1(Y(\CC) ,\QQ )$. As $\varphi^{-1}(T)$ is
 stabilized by ${\rm Gal}(\overline F /F )$ so is $\psi^M(U_{\QQ_p} )$. From
 Proposition~\ref{prop:constate} we deduce that $U$ is a Hodge substructure of $V^M$, so the torus $T$ is motivic.
 \end{proof}

From Theorem~\ref{thm:main} one also obtains the following compatiblility result with
automorphisms of $\CC$. The notation is as in Theorem~\ref{thm:main}.

\begin{remark} \label{rmk:tak}
 For a finitely generated subfield $F\subset \CC$ and  for $\sigma\in
\mathrm{Aut}(\CC)$ let $F^\sigma=\sigma(F)$. For a scheme $X$ separated and of finite type
over $F$ consider the $F^\sigma$-scheme $X^\sigma=X\otimes_\sigma F^\sigma $ with abelian
 fundamental group $\Pi^\sigma=\pi_1^{\ab}( X^\sigma(\CC ) )$, the pro-finite completion of which we denote by
 $\pi^\sigma$. Then $\sigma$ induces a continuous isomorphism $\pi\xrightarrow{\simeq} \pi^\sigma$.
  Let $\mathcal S\subset \Hom_\cont (\pi, \QQp^\times )$ be a  subset stabilized by an open
  subgroup of $G$ such that \eq{eq.autoc}{\varphi^{-1}( S )=\mathcal S \ {\rm where} \  S=\overline{\varphi(\mathcal S)}. \notag}
  Then $\mathcal S^\sigma=\sigma (\mathcal S)\subset  \Hom_\cont (\pi^\sigma, \QQp^\times )$ satisfies
  \eq{eq.autoc}{(\varphi^\sigma)^{-1}( S^\sigma )=\mathcal S^\sigma     \ {\rm where} \  S^\sigma=\overline{\varphi(\mathcal S^\sigma)}\notag} 
  and  $ \varphi^\sigma:  \Hom_\cont (\pi^\sigma, \QQp^\times )\to   \Mg^{\Pi^\sigma}_{\QQp}$ is the map
  analog to  $\varphi$.   Indeed, up to translation by a torsion character, the
  irreducible components of  $S^\sigma$ are the subtori 
  of $\Mg^{\Pi^\sigma}_{\QQp}$ induced by
  the maps $ \psi_{T}\otimes_{\sigma} \CC: X^{\rm reg} \otimes_\sigma \CC \to B\otimes_\sigma \CC$ of complex varieties, where $\psi_T$ is defined in Theorem~\ref{thm:main} 2). 
\end{remark}

\section{ Application to jumping loci  } \label{sec:jump}

Let the notation be as in Theorem~\ref{thm:jump}.
Denoting by $$\mathcal{L}: \Pi\to \CC[\Pi]^\times $$ the canonical $\CC[\Pi]$-valued character, $R\Gamma( X(\CC),\sF\otimes \mathcal{L})$
  is quasi-isomorphic to a
  bounded above complex $P$ of finitely generated free $\CC[\Pi]$-modules, see~
  \cite[Lemma~V.5,~10.13]{Bor84}. Moreover, for $L\in \Mg^\Pi_\CC(\CC)$ we get a
  quasi-isomorphism $R\Gamma (X(\CC) , \sF\otimes L)\simeq P\otimes_{\CC[\Pi],L} \mathbb C$.
  So by~\cite[Thm.~7.6.9]{EGAIII} the sets $\Sigma^i(\sF , j)$ are Zariski closed in $\Mg^\Pi_{\CC}(\CC )$
  for any $i,j\ge 0$.

  \begin{proof}[Proof of Theorem~\ref{thm:jump}]
 As $\sF$ descends to $D^b_c(X(\CC ) , K)$ for some number field $K\subset \CC$   see Definition \ref{defn:arithm})
 we see that
 $\Sigma^i(\sF , j)$ is defined by  a Zariski closed subset in $\Mg^\Pi_{K}$, which we
 denote by the same symbol. Note that  for  all but finitely
 many embeddings  $K\hookrightarrow \QQp$ the closure of each irreducible component of $S=\Sigma^i(\sF ,
 j)_{\QQp}$ in $\Mg^\Pi_{\ZZp}$ has non-empty special fibre.
 Choose  an embedding
 $K\hookrightarrow \QQp$ with this property   such that furthermore $\sF$ induces an arithmetic  $\sF_\et
 $ in $ D^b_c (X_{\CC,\et}, \QQp)$. After replacing $F$ by a finitely generated extension we
 can assume that each $\sigma \in \mathrm{Aut} (\CC /F)$ fixes $\sF_\et$ up to quasi-isomorphism.

 Clearly, $\varphi^{-1}(S)$ consists of those $L_\et \in \Hom_{\cont}(\pi,\QQp^\times)$,
 such that
 \[
\dim H^i(X_{\CC, \et},\sF_\et \otimes L_\et )> j
\]
and  each $\sigma \in \mathrm{Aut} (\CC /F)$ stabilizes this subset. So by
Theorem~\ref{thm:main} we observe that each irreducible component of $S$ is a torsion
translated motivic subtorus of $\Mg^\Pi_{\QQp}$, so the same is true for $\Sigma^i(\sF ,
j)$ over $\CC$.
\end{proof}

\section{Some remarks}\label{sec:rmks}

 Let $X$ be a separated scheme, which is geometrically connected and of finite type   over a finitely generated field
$F\subset \CC$. Let $G={\rm Aut}(\overline F/F)$ be the Galois group of $F$, where
$ \overline F$ is the algebraic closure of $F$ in $\CC$. Let $x_{0}\in X(F)$ be a fixed
$F$-rational point.
 Let $\Pi$ be the fundamental group $ \pi_1(X(\CC),x_0)$ and let $\pi$ be its pro-finite
 completion, which is isomorphic to $\pi_1^{\et}(X_{\overline F} , x_0 )$ and which is
 endowed with a continuous action of $G$.

 Let $K$ be a field. Let $\mathrm M_K^{\rm irr}$ be the moduli space of isomorphism classes of
 $K$-linear irreducible local systems  of rank $r$ on $X(\CC )$. For a prime number $p$ we obtain a map
 \[
\varphi: \Hom_\cont (\pi, {\rm GL}_r(\QQp ))^{\rm irr} \to \mathrm M^{\rm irr}_{\QQp}
 \]
from the irreducible continuous representations to the closed points of the moduli space.
Through its action on $\pi$ we have an induced action of $G$ on the domain of $\varphi$.
For $p'$ another prime number we obtain an analogous map $\varphi'$.
Let $\iota : \QQp \xrightarrow{\simeq} \overline\QQ_{p'}$ be an isomorphism of fields.

As a potential generalization of Corollary~\ref{cor:indp} we suggest the following problem.

\begin{question}\label{quest.inde}
Let $S$ be a closed subset of $\mathrm M^{\rm irr}_{\QQp}$ such that $\varphi^{-1}(S)$ is
stabilized by an open subgroup of $G$. Let $S'$ be an irreducible component of $S$ such that
$\varphi^{-1}(S')$ is non-empty.
\begin{itemize}
  \item[1)] Is the set of arithmetic points $s\in S'$ dense on $S'$? Recall that a closed
      point $s\in  \mathrm M^{\rm irr}_{\QQp}$ is called {\em arithmetic}  if $\varphi^{-1}(s)$
      is non-empty and stabilized by an open subgroup of $G$.  
\item[2)]
 Is $(\varphi')^{-1}(\iota (S'))$ \\
 a)  non-empty;\\
b) stabilized by an open subgroup of $G$?
\item[3)] In case $X$ is smooth and projective over $F$,
  is $\mathrm{RH}(S'(\CC ) ) $ a Zariski closed subset of the moduli space of irreducible rank $r$
  integrable connections, where $\rm RH$ is the Riemann-Hilbert correspondence?
\end{itemize}
\end{question}
Question~\ref{quest.inde}~1) was suggested to us by M.~Groechenig.
One can show (see  \cite[Section~3]{EG18}) that if $S$ consists of a single isolated point,  then a positive answer to
Question~\ref{quest.inde}~2) implies Simpson's integrality
conjecture~\cite[Conj.~5]{Sim90}.                 We also observe that the converse of 1),
  which would be a Andr\'e-Oort type question, has a negative answer: it is not the case
  that the Zariski closure of a set of torsion points is the union of motivic subtori
  translated by torsion points. For example, let $X$ be an elliptic curve defined over
  $F=\mathbb{Q}$,  so    the two-dimensional Hodge structure $H_1$ is irreducible. Then  $\Mg^\Pi_A$ is a two-dimensional torus, which contains one-dimensional subtori, which are thus non-motivic. 
However the set of torsion points  on such subtori is dense.

\section{An alternative approach }\label{sec.alter}

In this section we sketch an alternative  proof of a weaker version of
Theorem~\ref{thm:main},
which is however sufficient for our application to jumping loci in Theorem~\ref{thm:jump}.
This alternative proof is  short and does not use $p$-adic analysis, instead it is based
on  class field theory, more precisely the rank one   case of de Jong's conjecture~\cite{deJ01},
and the torus case of the Mordell-Lang conjecture~\cite{Lau84}. A similar technique is used by
Drinfeld in~\cite{Dr01}.

Let the notation be as in Section~\ref{sec:global}, i.e.\ $F\subset \mathbb C$ is a
finitely generated field, $X/F$ is a reduced separated scheme of finite type,
$\Pi=\pi_1^{\rm ab}(X(\CC))$, $\pi$ is the pro-finite completion of $\Pi$ and $\pi^{(p)}$
denotes its pro-$p$ completion.

Let $K$ be a number field. For each embedding $\iota:K\hookrightarrow \QQp$ we define a map
\[
 \varphi_\iota: {\rm Hom}_{\rm cont} ( \pi, \overline  \QQ^\times_p) \to  \Mg^\Pi_{K}
\]
as in Section~\ref{sec:torsion}.

Let $S\subset\Mg^\Pi_{K}$ be a closed subset. 
In the following we sketch an alternative argument for the following corollary of Theorem~\ref{thm:main}.
\begin{corollary}\label{car:altproof}
  Assume that $\Pi_\QQ$ has weights in $\{-2,-1\}$ and that for infinitely many embeddings $\iota:K\hookrightarrow \QQp$ there exist open subgroups
of $G$ which stabilize
$\varphi_\iota^{-1}(S)$.   Then $S_{\overline K}\subset  \Mg^\Pi_{\overline K}$ is a finite union of subtori
translated by torsion points.
\end{corollary}

\begin{proof}
By what is explained in Section~\ref{sec:global} we can assume without loss of generality
that $X$ is smooth over $F$.  We will show that the torsion points are dense in $S$. Then the corollary follows from
  the following version of the Mordell--Lang conjecture, proved by M.~Laurent~\cite[Thm.~2]{Lau84}.

  \begin{proposition}\label{prop.morlan}
Let $T$ be a split torus over $K$ and let $S\subset T$ be a closed subset such that
$T_{\rm tor}\cap S$ is dense in $S$. Then $S_{\overline K}$ is a finite union of subtori of
$T_{\overline K}$
translated by torsion points. 
    \end{proposition}

Let $V\subsetneq S$ be a closed subset. We want to find a torsion point on $S\setminus V$.
Let $\mathcal S$ be the closure of $S$ in $\Mg^\Pi_{\sO_K}$ and $\mathcal V$ be that of $V$. Choose a closed point $s\in
\mathcal S\setminus \mathcal V$  which satisfies the following two conditions
\begin{itemize}
  \item[1)]
    $\mathcal S$ is smooth over $\ZZ$ at $s$ and
    \item[2)] a
$\iota:K\hookrightarrow \QQp$ which induces the place of $K$ corresponding to the image of
$s$ in $\Spec(\sO_K)$ has the property that $\varphi_\iota^{-1}(S)$ is stabilized by an
open subgroup of $G$.
\end{itemize}

Then the completion  $R$ of
$\sO_{\mathcal S, s}$ is canonically a quotient of the Iwasawa algebra
$W(k(s))\llbracket \pi^{(p)}/{\rm
  tor}\rrbracket$   by an ideal stabilized by the action of  an open subgroup of $G$ on $\pi^{(p)}$.
Here $W(k(s))$ is the ring of Witt vectors of the residue field $k(s)$  of $s$. We may assume
without loss of generality that this open subgroup is equal to $G$. By smoothness in 1) we
see that
$\sO_{\mathcal S, s}$ is isomorphic to a formal power series ring $W(k(s))\llbracket
T_1,\ldots , T_d \rrbracket$.

By~\cite[Th.Fin.~1.9]{Del77} there exists a connected normal scheme $\mathcal W$ of finite type over
$\mathbb Z[1/p]$ with $k(\mathcal W )=F$ and a smooth scheme $f:\mathcal X\to \mathcal W$ with
generic fibre $X$ such that  $R^1f_* \ZZ/p\ZZ$ is locally constant and compatible with base
change. Then the action of $G$ on $\pi^{(p)}$ factors through $\pi_1(\mathcal W)$.

Choose a closed point $w\in \mathcal W$ and
let ${\rm Fr}_w\in \pi_1(\mathcal W)$ be the corresponding Frobenius. Class field theory
implies the finiteness of the Frobenius coinvariants $\pi^{(p)}_{\mathrm{Fr}_w}$  of the maximal $p$-adic quotient  of
the abelian \'etale fundamental group $\pi^{(p)}\cong \pi_1^{\rm ab}(\mathcal X_{\bar
  w})^{(p)}$, see~\cite[Thm.~1.3.1]{Del80}.
 De Jong's observed in~\cite[Sec.~3]{deJ01}  that this finiteness implies that the (non-zero)
quotient ring
\[
  \bar R= R /({\rm
    Fr}_w(x)-x | x\in R)= R/(T_1-{\rm Fr}_w(T_1), \ldots ,T_d-{\rm Fr}_w(T_d) )
\]
is finite  and flat over $\ZZ_p$. Indeed,  otherwise $\dim(\bar R/(p))$  would be $>0$ and
the canonical character \[\pi^{(p)}_{\mathrm{Fr}_w}\to (  \bar R/(p))^\times \] would have infinite image, contradicting
the above observation. 
A minimal prime ideal of $\bar
R$ finally gives rise to a torsion point of $S\setminus V$.
\end{proof}

 \end{document}